\documentclass[11pt, a4paper]{article}
\usepackage{xspace}

\usepackage{a4wide}
\usepackage{amsmath}
\usepackage{amssymb}
\usepackage{amsthm}
\usepackage{mathtools}
\usepackage{subfigure}
\usepackage{caption}
\usepackage[style=trad-alpha,url=true]{biblatex}
\usepackage[utf8]{inputenc}
\usepackage[normalem]{ulem}
\usepackage[usenames,dvipsnames]{xcolor}

\usepackage{tikz}
\usetikzlibrary{arrows,intersections}
\usepackage{tkz-euclide}
\usetkzobj{all}

\usepackage[ruled, linesnumbered]{algorithm2e}

\newtheorem{de}{Definition} 
 
\newtheorem{lemma}[de]{Lemma} 
\newtheorem{theo}[de]{Theorem}

\newcommand{\R}{\mathbb R}
\newcommand{\Z}{\mathbb Z}
\newcommand{\N}{\mathbb N}

\newcommand{\A}{\mathcal{A}}
\newcommand{\E}{\mathcal{E}}
\newcommand{\SD}{\mathcal{S}}
\newcommand{\D}{\mathcal{D}}

\newcommand{\RPTS}{\ref{RPT}($\SD'$)}
\newcommand{\Ed}{\E^{\text{dep}}}

\newcommand{\Es}{\E^{\text{start}}}
\definecolor{green}{rgb}{0,0.5,0}

\begin{document}  

\title{Finding Robust Periodic Timetables \\ by Integrating Delay Management}

\author{Julius Pätzold\\
Institute for Numerical and Applied Mathematics\\ University of Goettingen\\
Lotzestraße 16 - 18\\ 37083 Göttingen, Germany \\ j.paetzold@math.uni-goettingen.de}

\maketitle

\begin{abstract}
This paper defines and solves a mathematical model for finding robust periodic timetables by proposing an extension of the Periodic Event Scheduling Problem \eqref{PESP}. In order to model delayed and not nominal travel times already in the timetabling step, we integrate delay management into the periodic timetabling problem. After revisiting both \eqref{PESP} and delay management individually, we introduce a periodic delay management model capable of evaluating periodic timetables with respect to delay resistance. Having introduced periodic delay management, we define the Robust Periodic Timetabling problem \eqref{RPT}. Due to the high complexity of \eqref{RPT} we propose two different simplifications of the problem and introduce solution algorithms for both of them. These solution algorithms are tested against timetables found by standard procedures for periodic timetabling with respect to their delay-resistance. The computational results show that our algorithms yield timetables which can cope better with occurring delays, even on large-scale datasets and with low computational effort.
\end{abstract}

\section{Introduction}

Public transportation planning can be conceived of as a range of different subproblems that need to be solved: On a strategic level,  a public transportation networks need to be designed, operating lines need to be determined, and a timetable needs to be found. On an operational level, vehicles and crew need to be scheduled and finally delays have to be coped with. Each of these subproblems can be solved with respect to different objectives: From the viewpoint of the operating company, for example, the goal is to keep the costs of the public transport system low; passengers, on the other hand, want to have short travel times. Instead of viewing the planning process of public transportation as a whole, most research instead focuses on picking a certain subproblem, e.g., only line planning or only timetabling, and solving it individually. This process, however, merely leads to some local optimal solution. Nevertheless, recent work, e.g., \cite{petersen2013simultaneous, mariediss, Sch16, burggraeve2017integrating, PSSS17, PSS18} shows that the integration of subproblems is often superior to solving single problems sequentially. For example, \cite{PSS18} show how the integration of line planning, timetabling and vehicle scheduling leads to cost-optimal public transport plans that still can be computed efficiently. In this paper, though, we take on the passengers' view on a public transport system and concentrate on finding delay-resistant timetables. Delay-resistance is a key aspect when modeling a quality measure from a passengers' point of view. Nominal travel times, i.e., estimated travel times without delays, are certainly relevant for the attractiveness of traveling by train in general; Nevertheless, the amount of experienced delay (and uncertainty in general) impacts the passengers' reception on a public transport system even more, mostly in a negative way. Put differently, complaints about the nominal length of travel times are rare, whereas train punctuality is highly complained about and discussed by customers and press (see, e.g., \cite{connolly_2018}). Despite its high practical relevance, this aspect of considering delayed travel times is often neglected in periodic timetabling research. In our paper we take this factor into account and include a robustness measure -- motivated by delay management considerations -- and hence train punctuality into periodic timetabling. To meet this objective, we propose an approach to integrate the two problems of periodic timetabling and delay management by incorporating the concept of adjustable robust optimization.

The remainder of this paper is structured as follows: First, we introduce periodic timetabling and give a literature overview including how robustness concepts are applied to it. In the next section we define delay management and give a model for adapting it to periodic timetabling, hence defining a new evaluation of a periodic timetable. In Section \ref{sec:rpt}, we formulate the problem of finding robust periodic timetables and give two solution approaches based on iterative cutting-plane techniques. We reinforce the quality of our approaches via computational experiments in Section \ref{sec:exp}, and conclude the paper in Section \ref{sec:con}.

\section{Periodic Timetabling}\label{PTT-R}

Periodic timetabling is one of the most difficult problems in public transport optimization. Since its introduction in \cite{serafini1989mathematical} it has been studied extensively, see  \cite{odijk1996constraint, nachtigall1998periodic, peeters2003cyclic, liebchen2007periodic}. The problem first requires the definition of an event-activity-network.

\begin{de}
A periodic event-activity-network (EAN) is a directed graph $(\E, \A)$. The nodes $\E = \Ed\;\cup\; \E^{\text{arr}}$ are divided into departure and arrival events. Furthermore the activities are divided into the set of driving and waiting activities $\A^{dw}$ and the set of changing activities $\A^{ch}$, i.e., $\A=\A^{dw}\cup\A^{ch}$. We assume that $A^{dw}$ corresponds to train lines and hence forms a set of node-disjoint paths $\{l_1,...,l_n\}$ covering all events $\E$.
Each activity $a\in\A$ has lower and upper bounds $[L_a, U_a]$ for the allowed duration and a weight $w_a$ corresponding to the number of passengers traveling on that activity. For every event $i\in\E$ we are also given the number of passengers $w_i$ unboarding the train at event $i$ (for arrival events $i\in\E^{\text{arr}}$) or boarding the train at event $i$ (for departure events $i\in\Ed$).
\end{de}

Now the problem can be stated as follows.

\begin{de}[\cite{serafini1989mathematical}]
Given a periodic EAN ($\E,\A$), the Periodic Event Scheduling Problem assigns a time $\pi_e\in \N$ to each event $e\in\E$ in order to solve
\begin{align}\label{PESP}\tag{PESP}
\min \quad & \sum_{a=(i,j)\in \A}  w_a (\pi_j - \pi_i + z_a T)\\
\text{s.t.}\quad & \pi_j - \pi_i + z_a T \leq U_a \quad \forall a=(i,j)\in \A,\\
& \pi_j - \pi_i + z_a T \geq L_a \quad \forall a=(i,j)\in \A,\\
& z_a\in \Z \quad \forall a\in \A,\nonumber\\
& \pi_i \in \N \quad \forall i\in \E.\nonumber
\end{align}
$T\in \N$ denotes the time period, which is usually assumed to be $60$ minutes. For later use we define $\pi_a\coloneqq \pi_j - \pi_i + z_a T$ to be the time duration for all activities $a\in\A$ (with the summand $z_aT$ modeling the ``modulo $T$'' operation). We refer to a vector $\pi$ as a timetable. The set of all feasible timetables to an EAN is $\Pi = \Pi(\E, \A)$.
\end{de}
Several solution approaches have been proposed with the most prominent improvements being the cycle-base formulation (see \cite{peeters2001cycle, BorndoerferHoppmannKarbsteinetal18}), the modulo simplex algorithm (\cite{nachtigall2008solving, Goerigk-Sch-PESP}) and SAT-Formulations (\cite{grossmann2012solving}). Until today, there is an ongoing endeavor to improve on solving instances of \ref{PESP}.
Compared on the benchmark PESPlib (see \cite{pesplib}), the current best working solutions use a combination of several approaches to improve on the current best solutions: \cite{goerigk2017improved} combine modulo simplex and cycle-base formulations in an iterative manner. This paper's author extended \cite{PS16} to a general divide-and-conquer approach and applied the cycle-base formulation with several heuristics to improve on their solutions. Recently, \cite{BorndoerferLindnerRoth2019} combine Modulo Simplex, SAT, IP approaches and heuristics to deliver the current best solutions for PESPlib.
\medskip

When solving \eqref{PESP}, the objective is to minimize the sum of all passengers' travel times, which seems to be reasonable at first glance. In practice, however, there exist at least two problems for adapting a timetable found by solving \eqref{PESP}:

First, one does not know how passengers choose their routes in the EAN beforehand. For this problem, which formally boils down to the weights $w_a$ being variable instead of fixed, several attempts have been made to formulate an optimization problem that integrates the choice of passenger routes, see \cite{Goerigk-Siebert13, mariediss, SchmSch14, ttwr, gattermann2016integrating, borndorfer2017passenger}. 

The second problem is that there are usually delays in the network that need to be dealt with. The total amount of delay in a network may not be clear beforehand, but it is definitely unrealistic to completely neglect their existence and to optimize a timetable that is only guaranteed to work well in the nominal case of no delays (as is done when solving \ref{PESP}).
There exist several attempts to overcome this second problem of timetables being deficiently delay-resistant:

\cite{Goerigk15} applies the concept of recovery robustness for periodic timetabling, which was first introduced in \cite{liebchen2009concept} and extended in \cite{goerigk2014recovery} for aperiodic timetabling. It allows the modification of a solution after some scenario has been revealed. The aim of \cite{Goerigk15} is to minimize the cost of recovering to a solution over all scenarios. A heuristic approach for large instances is also given by solving a bicriteria optimization problem with the two linear objectives of travel time and robustness. A different approach to solve this problem is given in \cite{PBDM19}. The authors define a robust version of \eqref{PESP} by assuming that delays impact the interval $[L_a, U_a]$ for each activity and require a feasible robust timetable to be able to adjust its times in order to maintain feasibility for all activities. 
Another proposal for finding delay-resistant timetables makes use of stochastic optimization: In \cite{kroon2008stochastic}, the authors sample scenarios and try to optimize a rolled-out version of the timetable for all disturbances occurring in each of these samples. By restricting themselves to keep the cyclic train order fixed and only optimizing the slack times that can be allocated, the authors maintain tractability of the problem. Further research on robust periodic timetabling research includes \cite{kroon2007cyclic,LSSSP07,caimi2011periodic,maroti2017branch}. Next to robust periodic timetabling, there exist robust variants of the aperiodic timetabling problem, which is polynomially solvable in its nominal version. Literature includes \cite{fischetti2009fast,Atmos2010-goerigk, DHSS11, cacchiani2012lagrangian, lu2017improving}. For surveys on robustness in timetabling (periodic and aperiodic), see \cite{cacchiani2012nominal} and \cite{lusby2018survey}.

Interestingly, many of the mentioned attempts to robustify periodic timetabling view the problem as a three-stage process, consisting of 
\begin{center}
timetabling $\rightarrow$ scenario reveals $\rightarrow$ delay management. 
\end{center}
The third stage is just called differently: For example, \cite{Goerigk15} call it recovery of the timetable, \cite{PBDM19} call it adjustment by a linear decision rule and \cite{kroon2008stochastic} omit the third stage by mentioning that their model does not include traffic control decisions.
In this paper, we comprehend the third stage (delay management) as such by incorporating the original delay management problem as defined in the public transport optimization literature (see also next section). By doing so, we not only give a new model for robust periodic timetabling, but also model an integration of the periodic timetabling and the delay management problem. In the following we define the latter problem.

\section{Delay Management and Periodicity}

Delays constitute a major source of uncertainty when operating a bus or railway system. If a train is delayed, many rescheduling decisions have to be made, each of which may disturb the nominal schedule of a public transport system. The question of whether an otherwise punctual train should wait for a delayed feeder train in order to allow transferring passengers to reach their connection is known as delay management problem and has been studied extensively in the literature. The first papers dealing with this kind of question date back to \cite{Sch01,SBK01}. Integer programming models have been developed in \cite{Sch06,GHL05} and a recent survey about delay management models can be found in \cite{ORI}. 

Delay management is traditionally carried out in an aperiodic way. To this end, a periodic EAN and timetable can be rolled out for a certain time horizon in order to receive their aperiodic pendants. 

\begin{de}\label{de:rollout}
Let a periodic EAN ($\E,\A$), a time period $T$ and a timetable $\pi\in\Pi(\E,\A)$ be given. For a time horizon $[L,U]\subset\R$ we define an aperiodic EAN ($\E^*, \A^*$) by 
\begin{align*}
\E^*& \coloneqq \{(e,n)\in \E \times \Z \;| \;L\leq n \cdot T + \pi_e \leq U \},\\
\A^* & \coloneqq \{ (i,n), (j,m)\in \E^*\times \E^* \; |\; (i,j) = a \in \A \; \wedge\; L_a \leq \pi_j - \pi_i + (m-n) T \leq U_a \},
\end{align*}
and the aperiodic timetable $\pi^*_i = \pi_i + n T$ for all $(i,n)\in\E^*$.
\end{de}

To formulate an integer programming model for the delay management problem, we need to formally introduce delays. As is commonly done, we assume that a set of potentially expected source delays is known, e.g., caused by signaling problems, construction work, accidents, or bad weather conditions. These source delays cause propagated delays, e.g., for the same train at subsequent stations or for other trains that wait for the delayed train.
As in \cite{SchaSchoe09b} we allow two types of source delays: The first type is a delay $s_e\in\N$ at an event $e \in \E$ (e.g., staff being late for their shift) referring to a fixed point in time. The second type of source delay is a delay $s_a$ which increases the duration of an activity $a\in \A$, e.g., an increase of travel time between two stations due to construction work. Such a delay $s_a$ has to be added to the minimal duration $L_a$ of activity $a$.
If an event or activity has no source delay, we assume $s_e=0$ or $s_a=0$, respectively. We hence define the set of possible source delays for an EAN as
\begin{de}\label{de:unc}
Given an EAN (periodic or aperiodic) ($\E,\A$), define a set of scenarios $\SD$ as
\[
\SD(\E,\A)\coloneqq \{ s \in \R_{\geq 0}^{|\E|+|\A|} | \; s_i \leq \sigma \; \forall i\in \E\cup\A, \|s\|_1 \leq \rho \},
\]
with $\sigma \geq 0$ as the maximum single delay and $\rho\geq 0$ as the maximum sum of all source delays. We assume that there are no source delays on the change activities ($s_a=0$ for all $a\in\A^{ch}$).
\end{de} 

Hence, we define the uncertainty set similar to \cite{PBDM19}, which originates from \cite{bertsimas2011theory} that introduce a parameter (here, $\rho$) to regulate the ``budget of uncertainty'' which corresponds in this case to the amount of source delay $\rho$ that can be distributed to the activities $a\in\A$.

Having defined source delays, we can now state the integer programming formulation for the delay management problem. To model the wait-depart decisions, i.e., whether some train should wait for some other train at a station or not, we introduce binary variables
\[
y_a =
\begin{cases}
0 & \mbox{if changing activity } a \mbox{ is maintained,}\\
1 & \mbox{otherwise,}
\end{cases}
\]
for all changing activities $a \in \A^{ch}$.
The integer programming formulation then reads as follows:
\begin{de}[\cite{Sch01}]
Given an aperiodic EAN ($\E^*,\A^*$), an associated timetable $\pi\in\N^{|\E^*|}$, and source delays $s\in\SD(\E^*,\A^*)$, define the delay management problem as
\begin{align} \label{DM}
\tag{DM}&\min   & \sum_{i \in \E^{*\text{arr}}} w_i d_i &+ \sum_{a \in \A^{*ch}}  y_a w_a T  \\
&\text{s.t.}\quad &  d_i                  & \geq s_i  & \forall  &i \in \E^*, \label{DM-1}\\
&& \pi_j - \pi_i + d_j  - d_i & \geq  L_a + s_a & \forall & a=(i,j) \in \A^{*dw}, \label{DM-2}\\
&& M y_a + \pi_j - \pi_i + d_j - d_i & \geq  L_a  & \forall &a=(i,j) \in \A^{*ch},\label{DM-3}\\
&& d_i                  & \in   \R  & \forall &i \in \E^*, \nonumber \\
&& y_a                  & \in   \{0,1\}  & \forall  &a \in \A^{*ch} \nonumber.
\end{align}
The new timetable, called disposition timetable, is now defined as $(\pi_i + d_i)_{i\in\E^*}.$
\end{de}

Self-evidently, the delays $d_i$ have to be greater than the source delays at the respective events $i\in\E^*$. Then, for every driving or waiting activity $a\in\A^{*dw}$ the duration $\pi_j + d_j - (\pi_i + d_i)$ after disposition needs to be greater than the lower bound $L_a$ plus some possible source delays on that activity $s_a$ because we assume that the train cannot drive faster than that, i.e., \eqref{DM-2}. For the changing activities $a\in\A^{*ch}$ the model can decide if a change is maintained ($y_a=0$). In this case this activity has to satisfy the same inequality as all driving and waiting activities. If, on the other hand, the change is not maintained, then big $M$ in \eqref{DM-3} is triggered and the inequality does always hold. In that case, however, the objective function adds a full time period $T$ for every passengers that has missed the respective change. The objective function furthermore sums up the delay for all passengers up to the point at which they get off at their destination. This, of course, is only an approximation of reality. There exist more sophisticated models that take passenger rerouting into account, but they do so at the expense of a much more complicated and hence slower model.

The $d$-variables of model \eqref{DM} are defined slightly differently ($x_i\coloneqq d_i + \pi_i$ for $i\in\E^*$, see, e.g., \cite{Sch06}), but the above notation will be more convenient for later use.
Also note that \eqref{DM} does not consider upper bounds $U_a$ for activities $a\in\A$, which is reasonable since the duration $\pi_a$ of an activity plus the propagated delay $d_a$ should not be bounded by a model assumption, as this might lead to infeasibilities of the model, e.g., if $L_a + s_a > U_a$. Removing this assumption also from \eqref{PESP} would be a viable option, but is neglected here in order to maintain feasibility of the obtained timetable for \eqref{PESP}.% What instead makes sense to bound is the size of the exogene source delays.
\medskip

There exist several shortcomings of this delay management formulation: The passenger distribution $w$, for example, is not fixed in reality, and also penalizing a missed change with exactly $T$ minutes is not always correct. Nevertheless, \eqref{DM} can be regarded as a reasonable approximation for the delay management process, hence its establishment in the literature.
%Regardless of the establishment of model \eqref{DM} in the literature, it contains several flaws. One of them is that if a passenger travels from event $i$ to event $j$ by first taking train $A$ and then transferring to train $B$ via changing activity $a$, the model has difficulties when train $A$ is delayed by more than an hour, e.g. two hours. In this case, when train $B$ is on time and does not wait for train $A$, then $d_j$ will be $0$ and one hour (resp. time period $T$) will be added as penalty for missing change activity $a$.
%Second, if train $A$ is now delayed for one hour, train $B$ is on time and does not wait, and train $B$'s follow-up (scheduled one hour later) train $B'$ is delayed by one hour, then model \eqref{DM} only counts one hour of delay (because of the missed connection from train $A$ to train $B$) and not two hours (since train $B'$ was scheduled one hour later and is delayed for an additional hour).
%In general the weights $w_i$ need to be adjusted if some connection is missed, but this would lead to a difficult non-linear model.
In this paper we make use of the simplicity of \eqref{DM} by being able to fit a modified version of it into periodic timetabling. By doing so we are able to give an approximative evaluation of the behavior of a periodic timetable in (aperiodic) delay management, i.e., for \eqref{DM}.
The modification of \eqref{DM} to a periodic setting is the following: We assume that delays occur periodically and that, accordingly, the delayed timetable $d$ also needs to work periodically. This yields the subsequent model.
\begin{de}
Given a periodic EAN ($\E,\A$), a periodic timetable $\pi\in\Pi(\E, \A)$ and source delays $s\in\SD(\E,\A)$, we define the periodic delay management problem as
\begin{align} \label{P-DM}
\tag{P-DM}&\min    &\sum_{a \in \A} w_a d_a &+ {\sum_{i\in\Ed} w_i d_i}&\\
&\text{s.t.}\quad &  d_i                  & \geq  s_i  & \forall  &i \in \E \label{P-DM-1},\\
& & d_a &= d_j - d_i \quad &\forall &a=(i,j) \in \A^{dw}\label{P-DM-2},\\
&&  d_a &= d_j - d_i + z_a T & \forall &a=(i,j) \in \A^{ch}\label{P-DM-3},\\
&& \pi_a + d_a       & \geq  L_a + s_a & \forall & a=(i,j) \in \A \label{P-DM-4},\\
&& d_i                  & \in   \R  & \forall &i \in \E \cup \A, \nonumber\\
&& z_a                  & \in   \Z  & \forall  &a \in \A^{ch}, \nonumber
\end{align}
with the set of all feasible propagated delays $d$ being 
\[
\D = \D(\pi, s) \coloneqq \{d\in \R^{|\E| + |\A|}\; | \; \exists \, z\in\Z^{|\A^{ch}|} \text{ s.th. } (d,z) \text{ is feasible for \eqref{P-DM}} \}.
\]
\end{de}

At first glance, \eqref{P-DM} looks fairly similar to \eqref{DM}: The propagated delays should still be larger than the source delays, i.e., \eqref{P-DM-1}. In \eqref{P-DM-2} we then introduce propagated delays on activities by setting $d_a = d_j - d_i$, which could also be done for \eqref{DM}. The delay of a train adds up along driving and waiting activities (since these are executed by the same train) which is why we have to respect periodicity for delays $d_a$ only for changing activities, i.e., \eqref{P-DM-3}. A thorough example is given in Figure \ref{Big-example}. Note that this model assumes (similar to \eqref{DM}) that the choice of passenger paths is fixed. Otherwise a source delay of $T$ should not influence a periodic timetable because all passengers are able to take an earlier train which is then delayed by $T$ minutes. 

\begin{figure}[ht!]
\centering
\subfigure[Periodic EAN]{
    \resizebox{!}{0.26\textwidth}{
        \begin{tikzpicture}
  [auto=left,every node/.style={circle,fill=blue!40, minimum width=4em, scale=1.3}]
  \node (n1) at (-3,5)  {GÖ dep};
  \node (n2) at (0,5)  {H arr};
  \node (n3) at (3,5)  {H dep};
  \node (n4) at (6,5)  {HH arr};
    \node (n5) at (9,5)  {HH dep};
    \node (n6) at (12,5)  {HB arr};
   \node (n7) at (-3,0)  {OL dep};
  \node (n8) at (0,0)  {HB arr};
  \node (n9) at (3,0)  {HB dep};
  \node (n10) at (6,0)  {H arr};
    \node (n11) at (9,0)  {H dep};
      \node (n12) at (12,0)  {BS arr};
      %edges
      \draw[blue!90, -triangle 90] (n1) to node[draw=none, fill=none, label ={drive}, below] {} (n2);
        \draw[blue!90, -triangle 90] (n2) to node[draw=none, fill=none] {wait} (n3);
            \draw[blue!90, -triangle 90] (n3) to node[draw=none, fill=none] {drive} (n4);
        \draw[blue!90, -triangle 90] (n4) to node[draw=none, fill=none] {wait} (n5);
            \draw[blue!90, -triangle 90] (n5) to node[draw=none, fill=none] {drive} (n6);
                \draw[blue!90, -triangle 90] (n7) to node[below,draw=none, fill=none] {drive} (n8);
    \draw[blue!90, -triangle 90] (n8) to node[below, draw=none, fill=none] {wait} (n9);
    \draw[blue!90, -triangle 90] (n9) to node[below, draw=none, fill=none]  {drive}(n10);
        \draw[blue!90, -triangle 90] (n10) to node[below, draw=none, fill=none] {wait} (n11);
    \draw[blue!90, -triangle 90] (n11) to node[below, draw=none, fill=none]  {drive}(n12);
% Changes
    \draw[black!100, -triangle 90] (n2) to node[near start, below, left, draw=none, fill=none]  {change}(n11);
        \draw[black!100, -triangle 90] (n6) to node[near start, draw=none, fill=none]  {change}(n9);
                \draw[black!100, -triangle 90] (n10) to node[near end, above, right, draw=none, fill=none]  {change}(n3);
      \end{tikzpicture}
    }
}
\subfigure[Periodic Timetable with all $\pi_i,i\in\E$ and $\pi_a,a\in\A$]{
    \resizebox{!}{0.26\textwidth}{
        \begin{tikzpicture}
  [auto=left,every node/.style={circle,fill=blue!40, minimum width=4em, scale=1.3}]
  \node (n1) at (-3,5)  {0};
  \node (n2) at (0,5)  {15};
  \node (n3) at (3,5)  {18};
  \node (n4) at (6,5)  {30};
    \node (n5) at (9,5)  {33};
    \node (n6) at (12,5)  {50};
   \node (n7) at (-3,0)  {35};
  \node (n8) at (0,0)  {57};
  \node (n9) at (3,0)  {0};
  \node (n10) at (6,0)  {22};
    \node (n11) at (9,0)  {25};
      \node (n12) at (12,0)  {40};
      %edges
      \draw[blue!90, -triangle 90] (n1) to node[draw=none, fill=none] {\textcolor{black}{15}} (n2);
        \draw[blue!90, -triangle 90] (n2) to node[draw=none, fill=none] {\textcolor{black}{3}} (n3);
            \draw[blue!90, -triangle 90] (n3) to node[draw=none, fill=none] {\textcolor{black}{12}} (n4);
        \draw[blue!90, -triangle 90] (n4) to node[draw=none, fill=none] {\textcolor{black}{3}} (n5);
            \draw[blue!90, -triangle 90] (n5) to node[draw=none, fill=none] {\textcolor{black}{17}} (n6);
                \draw[blue!90, -triangle 90] (n7) to node[below, draw=none, fill=none] {\textcolor{black}{22}} (n8);
    \draw[blue!90, -triangle 90] (n8) to node[below, draw=none, fill=none] {\textcolor{black}{3}} (n9);
    \draw[blue!90, -triangle 90] (n9) to node[below, draw=none, fill=none]  {\textcolor{black}{22}}(n10);
        \draw[blue!90, -triangle 90] (n10) to node[below, draw=none, fill=none] {\textcolor{black}{3}} (n11);
    \draw[blue!90, -triangle 90] (n11) to node[below, draw=none, fill=none]  {\textcolor{black}{15}}(n12);
% Changes
    \draw[black!100, -triangle 90] (n2) to node[near start, below, left, draw=none, fill=none]  {10}(n11);
        \draw[black!100, -triangle 90] (n6) to node[near start, draw=none, fill=none]  {10}(n9);
                \draw[black!100, -triangle 90] (n10) to node[near end, above, right, draw=none, fill=none]  {56}(n3);
      \end{tikzpicture}
    }
}
\subfigure[Periodic Source Delays added, i.e., \textcolor{red}{$s_i$}, $i\in\E$ and \textcolor{red}{$s_a$}, $a\in\A$]{
    \resizebox{!}{0.26\textwidth}{
        \begin{tikzpicture}
  [auto=left,every node/.style={circle,fill=blue!40, minimum width=4em, scale=1.3}]
  \node (n1) at (-3,5)  {0};
  \node (n2) at (0,5)  {15};
  \node (n3) at (3,5)  {18};
  \node (n4) at (6,5)  {30};
    \node (n5) at (9,5)  {33};
    \node (n6) at (12,5)  {50};
   \node (n7) at (-3,0)  {35};
  \node (n8) at (0,0)  {57};
  \node (n9) at (3,0)  {0};
  \node (n10) at (6,0)  {22};
    \node (n11) at (9,0)  {25\textcolor{red}{+10}};
      \node (n12) at (12,0)  {40};
      %edges
      \draw[blue!90, -triangle 90] (n1) to node[draw=none, fill=none] {\textcolor{black}{15}\textcolor{red}{+8}} (n2);
        \draw[blue!90, -triangle 90] (n2) to node[draw=none, fill=none] {\textcolor{black}{3}} (n3);
            \draw[blue!90, -triangle 90] (n3) to node[draw=none, fill=none] {\textcolor{black}{12}\textcolor{red}{+2}} (n4);
        \draw[blue!90, -triangle 90] (n4) to node[draw=none, fill=none] {\textcolor{black}{3}} (n5);
            \draw[blue!90, -triangle 90] (n5) to node[draw=none, fill=none] {\textcolor{black}{17}\textcolor{red}{+5}} (n6);
                \draw[blue!90, -triangle 90] (n7) to node[below, draw=none, fill=none] {\textcolor{black}{22}} (n8);
    \draw[blue!90, -triangle 90] (n8) to node[below, draw=none, fill=none] {\textcolor{black}{3}} (n9);
    \draw[blue!90, -triangle 90] (n9) to node[below, draw=none, fill=none]  {\textcolor{black}{22}}(n10);
        \draw[blue!90, -triangle 90] (n10) to node[below, draw=none, fill=none] {\textcolor{black}{3}\textcolor{red}{+7}} (n11);
    \draw[blue!90, -triangle 90] (n11) to node[below, draw=none, fill=none]  {\textcolor{black}{15}}(n12);
% Changes
    \draw[black!100, -triangle 90] (n2) to node[near start, below, left, draw=none, fill=none]  {10}(n11);
        \draw[black!100, -triangle 90] (n6) to node[near start, draw=none, fill=none]  {10}(n9);
                \draw[black!100, -triangle 90] (n10) to node[near end, above, right, draw=none, fill=none]  {56}(n3);
      \end{tikzpicture}
    }
}
\subfigure[Periodic Disposition Timetable, i.e., \textcolor{green}{$\pi_i + d_i$}, $i\in\E$ and \textcolor{green}{$\pi_a + d_a$}, $a\in\A$]{
    \resizebox{!}{0.26\textwidth}{
        \begin{tikzpicture}
  [auto=left,every node/.style={circle,fill=blue!40, minimum width=4em, scale=1.3}]
  \node (n1) at (-3,5)  {\textcolor{green}{0}};
  \node (n2) at (0,5)  {\textcolor{green}{23}};
  \node (n3) at (3,5)  {\textcolor{green}{26}};
  \node (n4) at (6,5)  {\textcolor{green}{40}};
    \node (n5) at (9,5)  {\textcolor{green}{43}};
    \node (n6) at (12,5)  {\textcolor{green}{65}};
   \node (n7) at (-3,0)  {\textcolor{green}{35}};
  \node (n8) at (0,0)  {\textcolor{green}{57}};
  \node (n9) at (3,0)  {\textcolor{green}{0}};
  \node (n10) at (6,0)  {\textcolor{green}{22}};
    \node (n11) at (9,0)  {\textcolor{green}{35}};
      \node (n12) at (12,0)  {\textcolor{green}{50}};
      %edges
      \draw[blue!90, -triangle 90] (n1) to node[draw=none, fill=none] {\textcolor{green}{23}} (n2);
        \draw[blue!90, -triangle 90] (n2) to node[draw=none, fill=none] {\textcolor{green}{3}} (n3);
            \draw[blue!90, -triangle 90] (n3) to node[draw=none, fill=none] {\textcolor{green}{14}} (n4);
        \draw[blue!90, -triangle 90] (n4) to node[draw=none, fill=none] {\textcolor{green}{3}} (n5);
            \draw[blue!90, -triangle 90] (n5) to node[draw=none, fill=none] {\textcolor{green}{22}} (n6);
                \draw[blue!90, -triangle 90] (n7) to node[below, draw=none, fill=none] {\textcolor{green}{22}} (n8);
    \draw[blue!90, -triangle 90] (n8) to node[below, draw=none, fill=none] {\textcolor{green}{3}} (n9);
    \draw[blue!90, -triangle 90] (n9) to node[below, draw=none, fill=none]  {\textcolor{green}{22}}(n10);
        \draw[blue!90, -triangle 90] (n10) to node[below, draw=none, fill=none] {\textcolor{green}{13}} (n11);
    \draw[blue!90, -triangle 90] (n11) to node[below, draw=none, fill=none]  {\textcolor{green}{15}}(n12);
% Changes
    \draw[black!100, -triangle 90] (n2) to node[near start, below, left, draw=none, fill=none]  {\textcolor{green}{12}}(n11);
        \draw[black!100, -triangle 90] (n6) to node[near start, draw=none, fill=none]  {\textcolor{green}{55}}(n9);
                \draw[black!100, -triangle 90] (n10) to node[near end, above, right, draw=none, fill=none]  {\textcolor{green}{4}}(n3);
      \end{tikzpicture}
    }
}
\caption{Periodic Delay Management}
\label{Big-example}
\end{figure}
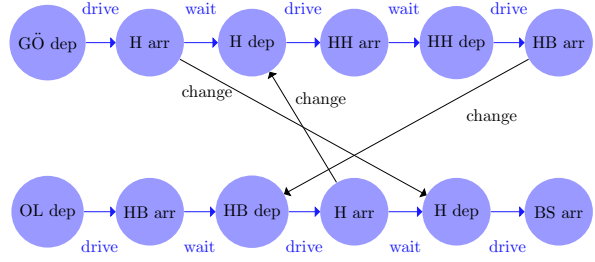
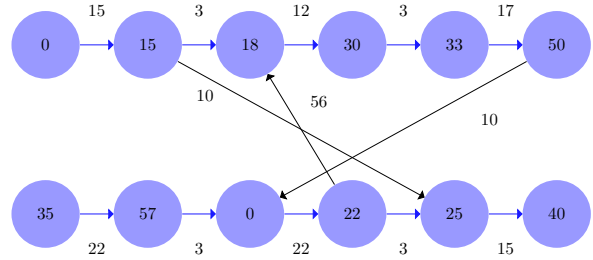
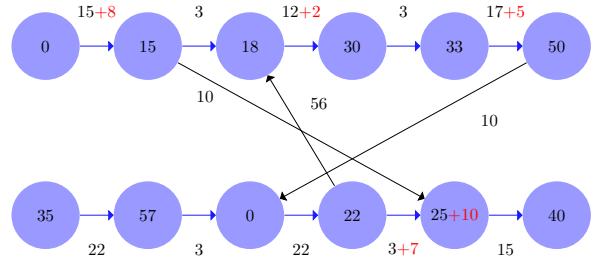
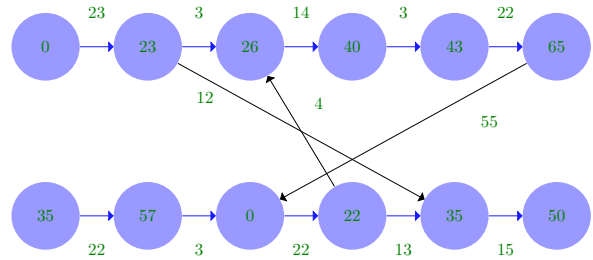

In the objective function we sum up the delayed activities and delayed departure events, which corresponds to summing up the delay on arrival events, since for every node-path ($e_1, ..., e_n$) in the EAN it holds that
\[
d_{e_n} = d_{e_1} + \sum_{i=2}^n d_{e_i} - d_{e_{i-1}} = d_{e_1} + \sum_{i=2}^n d_{(e_i,e_{i-1})}.
\]
By summing up along the activities (and departure events) and due to the periodicity we do not have to deal with missed changes separately, since a missed change is modeled by the periodicity as a long-lasting change activity (cf. (d) in Figure \ref{Big-example}). Whereas this provides a way to cope with cases of one train being delayed for more than one hour, the periodicity shrinks the space of source scenarios: In model \eqref{P-DM} every repetition of a train is assumed to have the same source delay and for each repeating train the same delay management strategy has to be chosen. Then again, it can be reasoned that it is of course possible to consider not only one but many different scenarios and to choose a timetable that is robust (or delay-resistant) to all of them. In Section \ref{sec:exp} we will discuss whether a periodic timetable found by minimizing \eqref{P-DM} is also delay-resistant when rolled out to an aperiodic network and evaluated by \eqref{DM}.
As a side note, see that \eqref{P-DM} can also be used to model disturbances like construction work: In this case source delays occur periodically and a revised periodic timetable needs to be found. 

The problem with delay management models by themselves, be it \eqref{DM} or \eqref{P-DM}, is that the nominal timetable is fixed and hence the model can only minimize the delays $d$ and not overall travel times $\pi + d$. This fact hence leads to the idea of combining the two objectives of nominal travel time $\pi$ and delayed times $d$ into a problem that evaluates a timetable by its overall behavior, meaning nominal travel time plus delays.

\section{Robust Periodic Timetabling}\label{sec:rpt}

In this section we now define a problem that allows us to find delay-resistant timetables. To this end, we make use of the concept of adjustable robustness (cf. \cite{ben2004adjustable}) which arises in robust optimization (see, e.g., \cite{ben2009robust, GoeSchoe13-AE}) to define the following problem.

\begin{de}
Given a periodic EAN ($\E,\A$), we search for a timetable $\pi\in\Pi$ that has the best worst-case behavior with respect to all scenarios $s\in\SD=\SD(\E,\A)$ that are likely to happen. Formally, we want to solve the problem
\begin{equation}\label{RPT}\tag{RPT}
\min_{\pi\in\Pi}\,\sup_{s\in\SD}\min_{d\in\mathcal{D}(\pi, s)} \tau(\pi,d)
\end{equation}
with
\[
\tau(\pi,d)\coloneqq \sum_{a\in\A} w_a (\pi_a + d_a) + {\sum_{i\in\Ed} w_i d_i}.
\]
\end{de}
The intention behind minimizing a timetable $\pi$ against its worst-case scenario $s\in\SD$ is that we want to require robustness against ``small'' perturbations of the nominal timetable. These perturbations $\SD$ are modeled by choosing $\sigma$ and $\rho$ to be rather small parameters. Details can be found in Table \ref{tab:data} in Section \ref{sec:exp}.

Of course, solving \eqref{RPT} is quite an ambitious endeavor: In Section \ref{PTT-R} we already mentioned the intrinsic difficulty of solving \eqref{PESP} in itself. In \eqref{RPT}, however, an additional layer of difficulty is added since, even if a timetable $\pi$ is fixed, the remaining sup-min problem of determining the worst-case scenario is a discrete bilevel-optimization problem (or, more precisely, a continuous-discrete bilevel problem, as $\SD$ is continuous, but the $z-$variables within the constraints of $\D$ are discrete), for which there is no known procedure to generally solve it to optimality (cf. \cite{sinha2018review}). %{Even the existence of a worst-case scenario $s\in\SD$ for a given timetable might not be guaranteed in general.} 
Nonetheless, for \eqref{RPT} in general we can at least show the existence of a minimal timetable $\pi\in\Pi$.

\begin{lemma}
There exists an optimal solution $\pi\in\Pi$ to \eqref{RPT}.
\end{lemma}
\begin{proof}
The objective function $\tau$ of \eqref{RPT} is linear in $d\in\D(\pi,s)$ as it just sums up the $d$-variables. The infimum of concave functions (i.e., especially linear functions) is again concave. Hence $\varphi(\pi,s) := \min_{d\in\mathcal{D}(\pi, s)} \tau(\pi,d)$ is concave. From the concavity of $\varphi$ and compactness of $\SD$ we derive that $\sup_{s\in\SD} \varphi(\pi,s)$ is finite. Since $\Pi$ is a finite set (all $\pi_i$ can be assumed to lie in $[0,T-1]$ due to periodicity), we can enumerate all timetables and determine the one with the smallest supremum. This yields the minimal solution to \eqref{RPT}.
\end{proof}

Thus, there exists an optimal timetable $\pi$ (although there might not exist a corresponding worst-case scenario $s\in\SD$ since $\varphi(\pi,s)$ might not be continuous due to the $z$-variables hidden in $\D$), but it is highly unlikely to determine it in reasonable time. We hence need to simplify \eqref{RPT} in order to achieve our goal of finding robust periodic timetables. To this end, we propose two different simplifications.

\begin{itemize}
\item[(A)] We assume that the strategy for delay management can be expressed in terms of timetable and scenario. This can be done, for example, by enforcing a no-wait policy for trains. With the fixed strategy we can transform the inner max-min problem of finding the worst-case scenario to a mixed-integer maximization problem. The remaining min-max problem can then be solved by a cutting-planes approach for robust optimization problems.
\item[(B)] We find a solution to the inner max-min problem heuristically, e.g., by sampling scenarios until a bad-case (not worst-case) scenario is found. We can then solve an integrated timetabling-delay-management problem with respect to some finite scenario set $\SD'\subset \SD$ and iteratively increase the scenario set until we are not able to find a worse scenario for the currently best timetable.
\end{itemize}

We can show that the first simplification, if solved to optimality, leads to an upper bound of \eqref{RPT}, whereas the second simplification leads to a lower bound on \eqref{RPT}. In the following we describe these two approaches in more detail.

\subsection{Simplification A: Fixed Delay Management Strategy}

To carry out the first approach, we fix our delay management strategy. In this paper, we chose the no-wait strategy (which was successfully done in \cite{LSSSP07} for minimizing the expected delay of a periodic timetable), meaning that no train waits for delayed passengers from other trains. Thus, delays are only propagated along driving and waiting activities, leading to the following form of $d$ (see also \cite{anitahabil}). Formally:

\begin{de}\label{de:no-wait}
Assume an EAN ($\E,\A$), a timetable $\pi$, and some source delays $s$ and let $\Es\subset\E$ be the set of events in the EAN such that every $e\in\Es$ has no incoming driving or waiting activity. The \emph{no-wait strategy} for delay management returns a solution $d$ to \eqref{P-DM} such that
\[
 d_j =
\begin{cases}
s_j, & j\in\Es,\\
\max \{d_i + s_a - (\pi_a - L_a), s_j\}, & (i,j)=a\in \A^{dw},
    \end{cases}
\]
holds for all events $i\in\E$. We denote the set of all feasible $d$ for which this property holds as $\D'(\pi,s)\subseteq\D(\pi,s)$.
\end{de}
The explanation of this definition is as follows: The first event of every line of the underlying line concept has no incoming driving or waiting activity. Hence, there is no delay that could possibly be propagated. Accordingly, we can set the propagated delay to its source delay, i.e., $d_i = s_i$ for $i\in\Es$. Now we propagate this delay along the line: The next event $j$ has an incoming driving activity $a=(i,j)$. We know that event $i$ is delayed by $d_i$ and that we have $\pi_a - L_a$ slack on the activity (meaning that we can save $\pi_a-L_a$ time by driving faster), but there also exists a source delay $s_a$ on this driving activity. By utilizing the slack of activity $a$ we can hence change the propagated delay to $d_i + s_a - (\pi_a - L_a)$ units of time. After having saved time, we encounter the source delay $s_j$ on event $j$ which leads to a propagated delay of $d_j = \max\{d_i + s_a - (\pi_a - L_a), s_j\}$ since $d_j$ is constrained by $d_j \geq s_j$ (from \eqref{P-DM}, i.e., we are not allowed to schedule events earlier than in the nominal timetable). Now the propagated delays $d_j$ are chosen such that the train drives along its line as fast as possible, potentially ignoring passengers on change activities, and thereby yielding the no-wait strategy.

We now can formulate a program that determines a scenario $s\in\SD$ and show that this program finds a scenario that is indeed a worst-case scenario for the no-wait strategy.

\begin{de}\label{de:f-wc}
Given a timetable $\pi\in\Pi$, we can find a scenario $s\in\SD$ regarding the no-wait strategy by solving
\begin{align}\label{F-WC}\tag{F-WC($\pi$)}
& \max & & \tau(\pi,d)\\
& \text{s.t.} & d_i &= s_i \quad &&\forall i \in \Es, \label{F-WC-1}\\
& & d_j &\geq d_i + s_a - (\pi_a - L_a)\quad &&\forall a=(i,j) \in \A^{dw},\label{F-WC-2}\\
& & -M z_a + d_j &\leq d_i + s_a - (\pi_a - L_a) \quad &&\forall a=(i,j) \in \A^{dw},\label{F-WC-3}\\
& & d_j &\leq M(1 - z_a) + s_j \quad &&\forall a=(i,j) \in \A^{dw}, \label{F-WC-4}\\
& & \pi_a + d_a&\leq L_a +  T  - 1\quad &&\forall a \in \A^{ch},\label{F-WC-5}\\
& & z_a& \in\{0,1\} &&\forall a\in\A^{dw},\nonumber\\
& & d & \in \D(\pi, s), && \nonumber \\
& & s & \in S. && \nonumber
\end{align}
\end{de}

\begin{theo}
A scenario found by solving \eqref{F-WC} yields a worst-case scenario for the no-wait strategy, i.e.,
\[\label{theo-fwc}
\max_{s\in\SD}\min_{d\in\D'(\pi,s)} \tau(\pi,d)  \Leftrightarrow \eqref{F-WC}.
\]
\end{theo}

\begin{proof}
First, it is worth mentioning that \eqref{F-WC} is indeed a mixed-integer linear program: The set $\SD$ has only linear constraints and also $d\in\D(\pi, s)$ is linear as it is only a short way of repeating constraints \eqref{P-DM-1}-\eqref{P-DM-4}. For showing the equivalence we rewrite the left side of \eqref{theo-fwc} with the help of Definition \ref{de:no-wait}:
\begin{align*}
\max_{s\in\SD}\min_{d\in\D'(\pi,s)}\tau(\pi,d) & \Leftrightarrow & \label{inter} \tag{*}\max_{s\in\SD} \quad & \min_{d\in\D(\pi,s)} \tau(\pi,d) & \\
& & \text{s.t.} \quad & d_i = s_i& \forall & i\in\Es,\\
& & & d_j = \max\{d_i + s_a - (\pi_a - L_a), s_j\} & \forall & (i,j)=a\in\A^{dw}.
\end{align*}
Now we can linearize the $\max$ operator by introducing a big $M$ and auxiliary variables $z_a$:
\begin{align*}
\eqref{inter}\quad \Leftrightarrow \quad \max_{s\in\SD} \quad & \min \;\tau(\pi,d)\tag{**}\\
\text{s.t.}\quad & d_i = s_i &\forall & i\in\Es \\
& d_j \leq (1-z_a) M + s_j &\forall & (i,j)=a\in\A^{dw},\\
& d_j \geq d_i + s_a - (\pi_a - L_a) &\forall & (i,j)=a\in\A^{dw},\\
&-z_a M + d_j  \leq d_i + s_a - (\pi_a - L_a)& \forall &(i,j)=a\in\A^{dw},\\
& z_a\in\{0,1\} &\forall &a\in \A^{dw},\\
& d\in\D(\pi,s). & &
\end{align*}
Note that $d_j\geq s_j \; \forall\; j\in\E$ is ensured by $d\in\D(\pi,s)$.

Finally, we can restrict $\pi_a + d_a \in [L_a, L_a + T - 1]$ for all changing activities $a\in\A^{ch}$ since the $d_a$ are minimized and do not impact any other constraints. Hence
\[ \tag{\ref{F-WC-5}}
\pi_a + d_a \leq L_a + T - 1 \quad \forall a \in \A^{ch}
\]
can be added to the model. With this final constraint at hand it is left to show that the $\min$ operator can be dropped. This is the case because at this point the $d$- and $z$-variables are determined uniquely for every $s\in\SD$. 

To show this, consider that if we fix some $s\in\SD$, the $d_i$ for $i\in\Es$ are determined to be equal to $s_i$. Then, starting from these start events $\Es$, the delays propagate along all driving and waiting activities uniquely via 
$d_j =\max\{d_i + s_a - (\pi_a - L_a), s_j\}$ for all $(i,j)=a\in\A^{dw}$. Hence, all $d_i$ for $i\in\E$ and also $d_a$ (and $z_a$) for $a\in\A^{dw}$ are determined uniquely. Additionally, \eqref{F-WC-5} determines $d_a$ (and $z_a$) for all $a\in\A^{ch}$.

Thus, the inner minimum can be dropped as (by uniqueness) there exists exactly one feasible set of variables $(d,z)\in\E\times\A\times\A$ leading to $(**)\Leftrightarrow\eqref{F-WC}$ and thereby to the Theorem's statement.
\end{proof}
Note that in \eqref{F-WC} we can write $\max$ instead of $\sup$ as opposed to in the notation in \eqref{RPT}: Since the sum of all source delays is bounded, the propagated delays are also bounded. Hence we can estimate an upper bound on the modulo parameters $z_a$ for all $a\in\A^{ch}$ (because the difference $d_j - d_i$ is also bounded). In doing so, we can enumerate all possible combinations of values for integer variables. For each of these enumerations we solve a linear optimization problem with no integer variables and the maximal of these finite number of different solutions is the maximal solution.

With the knowledge that \eqref{F-WC} yields a worst-case scenario for a timetable $\pi$ -- keeping in mind that the delay management strategy is fixed -- we can define the following reduced version of the overall problem \eqref{RPT}.

\begin{de}
Assuming the no-wait strategy, problem \eqref{RPT} reduces to
\begin{equation}\label{F-RPT}\tag{F-RPT}
\min_{\pi\in\Pi} \text{F-WC($\pi$)}.
\end{equation}
which can be reformulated as
\begin{align*}\tag{F-RPT($\SD$)}
& \min & & t\\
& \text{s.t.}
& t &\geq \tau(\pi,d_s) \quad & \forall s \in \SD,\\
& & d_s &\in \D'(\pi,s) & \forall s \in \SD,\\
& & \pi &\in \Pi,\; t\in\R. &
\end{align*}
\end{de}

If $|\SD|$ is finite, then \eqref{F-RPT} is a linear mixed-integer minimization problem, which is known to be solvable. If $|\SD|$ is infinite, on the other hand, \eqref{F-RPT} can be solved via cutting planes, as illustrated in Algorithm \ref{f-cp} and Figure \ref{basic-scheme}. A thorough investigation of cutting plane algorithms for min-max problems and convergence proofs are given in \cite{PS18}. The authors also propose speed-up techniques for the cutting plane approach making use of the fact that the single robustification and pessimization steps do not necessarily need to be solved to optimality. We used some of these insights in our implementation and will explain them further in Section \ref{sec:exp}.

\begin{algorithm}[h!]
	\DontPrintSemicolon
	\KwIn{EAN ($\E,\A$), nominal scenario $s_{\text{nom}} \in \SD$, stopping criterion $\epsilon > 0$
}
	\KwOut{$\epsilon$-optimal solution $\pi\in \Pi(\E,\A)$ to \eqref{F-RPT}, upper bound $ub_k$ to \eqref{RPT}}
	$\SD_0 \gets \{s_{\text{nom}}\}$, $k \gets 0$, ${lb}_0 \gets -\infty$, ${ub}_0\gets \infty$\;
	\While{$ub_k - lb_k > \epsilon$}{
%        \medskip
		$(\pi_k, d) \gets$ solution to F-RPT($\SD_k$)\;
	    $lb_{k+1} \gets \tau(\pi_k,d)$ \;
%		\medskip
        $(s_k, d)\gets$ solution to F-WC($\pi_k$)\;
        ${ub}_{k+1}\gets \min(ub_k, \tau(\pi_k,d))$\;
%        \medskip
        $\SD_{k+1} \gets \SD_k \cup \{ s_k\}$\; 
 		$k \gets k+1$\; 
		}
\Return $\pi_k, ub_k$\;
\caption{\textbf{Cutting Plane Approach}}
\label{f-cp}
\end{algorithm}

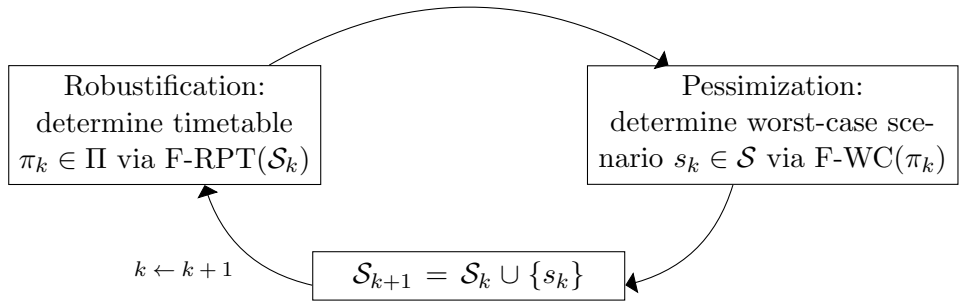
\begin{figure}[h!]
\centering
\begin{tikzpicture}[scale=0.4,auto=left,every node/.style={draw, rectangle, scale=1.0}]
  \node[text width = 12em, text centered, align=center] (pes) at (10,0)  {Pessimization: \\ determine worst-case scenario $s_k\in \SD$ via F-WC($\pi_k$)};
  \node[text width = 10em, text centered, align=center] (opt) at (-10,0)  {Robustification: \\ determine timetable $\pi_k\in \Pi$ via F-RPT($\SD_k$)};
  \node[text width = 10em, text centered, align=center] (wc) at (0,-5) {$\SD_{k+1} = \SD_k\cup\{s_k\}$};
     \draw[bend left, -triangle 90] (wc) to node[draw=none, fill=none] {\scriptsize{$k\gets k+1$}} (opt);
      \draw[bend left, -triangle 90] (opt) to node[draw=none, fill=none] {} (pes);
          \draw[bend left, -triangle 90] (pes) to node[draw=none] {} (wc);
\end{tikzpicture}
\caption{Cutting Plane Approach for Simplification A}\label{basic-scheme}
\end{figure}

%\begin{lemma}
%For every $\epsilon>0$, Algorithm \ref{f-cp} finds a solution with objective $z^*$ such that for the optimal objective value $z^{opt}$ of F-RPT($\SD$) it holds $\|z^* - z^{opt}\| \leq \epsilon$.
%\end{lemma}
%\begin{proof}
%As shown in \cite{PS18} in Lemma 5, we get convergence of Algorithm \ref{f-cp} already by finiteness of $\Pi$, since we can restrict $\pi_i$ to the integers $\{0,...,T-1\}$ for all $i\in\E$.
%\end{proof}

With \eqref{F-RPT} we have found a simplification of \eqref{RPT}, which is not only solvable, but also yields an upper bound for \eqref{RPT}.

\begin{lemma}
\eqref{F-RPT} is an upper bound for \eqref{RPT}.
\end{lemma}
\begin{proof}
It holds that
\begin{align*}
\eqref{F-RPT} \Leftrightarrow \min_{\pi\in\Pi} \eqref{F-WC} & \Leftrightarrow \min_{\pi\in\Pi}\max_{s\in\SD}\min_{d\in\D'(\pi,s)}\tau(\pi,d) \\
&\leq \min_{\pi\in\Pi}\sup_{s\in\SD}\min_{d\in\D(\pi,s)}\sum_{a\in\A}\tau(\pi, d) \Leftrightarrow \eqref{RPT}.
\end{align*}
\end{proof}

Hence we have found a relaxation of \eqref{RPT}. Before investigating it computationally in Section \ref{sec:exp}, we focus on the second approach that gives us a lower bound on \eqref{RPT}.

\subsection{Simplification B: Finding Bad-Case Scenarios}

If we do not want to restrict ourselves to a fixed delay management strategy as in Simplification A, we can alternatively shrink the space $\SD$ and consider only a finite number of scenarios. This problem yields a mixed-integer program:

\begin{lemma}
\RPTS~with finite $|\SD'|$ is a mixed-integer program.
\end{lemma}
\begin{proof}
We start by rewriting
\begin{align*}
\eqref{RPT}(\SD') \; \Leftrightarrow \; \min_{\pi\in\Pi}\max_{s\in\SD'} \min_{d\in\D(\pi,s)} \tau(\pi,d) \, \Leftrightarrow \quad & \min & t & & \tag{*}\\
& \text{s.t.} & t & \geq \min_{d\in\D(\pi,s)}\tau(\pi,d)  & \forall s \in \SD',\\
& & \pi & \in \Pi,&\\
& & t &\in\R, &
\end{align*}

and by creating $|\SD'|$ duplicates of $d$-variables we get

\begin{align*}
(*) \quad \Leftrightarrow \quad & \min & t & & \\
& \text{s.t.} & t & \geq \tau(\pi,d_s)  & \forall s \in \SD',\\
& & d_s & \in \D(\pi,s)& \forall s\in\SD',\\
& & \pi & \in \Pi,&\\
& & t & \in \R.&\\
\end{align*}
Linearity of $\pi$ and $d_s$ in $\tau$ and $\D$ yield the statement.
\end{proof}

Solving this problem furthermore gives a lower bound on \eqref{RPT}.

\begin{lemma}
\eqref{RPT}($\SD'$) is a lower bound on \eqref{RPT}($\SD$) if $\SD'\subseteq\SD$.
\end{lemma}
\begin{proof}
Fix a $\pi\in\Pi$. We get 
\[
\max_{s\in\SD'}\min_{d\in\D(\pi,s)}\tau(\pi,d) \leq \max_{s\in\SD}\min_{d\in\D(\pi,s)}\tau(\pi,d),
\]
and since this holds for each $\pi\in\Pi$ the above result is true.
\end{proof}

The basic idea of our proposed solution algorithm to this second simplification is quite similar to the cutting plane algorithm in the previous subsection: We solve a relaxed version of \eqref{RPT}, namely \eqref{RPT} on the finite scenario set $\SD_k\subset \SD$. Once we are given a solution $\pi_k$ to this problem, we try and find a new scenario $s_k$. Desired properties for this scenario are that it is not contained in $\SD_k$ and that \eqref{P-DM} on $(\pi, s_k)$ yields a worse objective value than \eqref{P-DM} on $(\pi,s)$ for any $s\in\SD_k$. As mentioned earlier in this paper, an exact algorithm for determining a worst-case scenario has not yet been found. In our implementation, we will try to find worst-case scenarios by sampling many scenarios and solving \eqref{P-DM} on each of them. This idea is summarized in Algorithm \ref{f-cp-2} and Figure \ref{basic-scheme-2}.

\begin{algorithm}[h!]
	\DontPrintSemicolon
	\KwIn{EAN ($\E,\A$), nominal scenario $s_{\text{nom}} \in \SD$, number of iterations $N\in\N$, number of sampled scenarios $M\in\N$ 
}
	\KwOut{Optimal Solution $\pi_k$ to \eqref{RPT}($\SD_k$) with $\SD_k\subseteq\SD$, lower bound $lb$ to \eqref{RPT}}
	$\SD_1 \gets \{s_{\text{nom}}\}$\;
	\For{$k=1,\dots,N$}{
		$(\pi_k, d^*) \gets $ solution to RPT($\SD_k$)\;
	    $lb \gets \tau(\pi_k, d^*)$\;
	    $ub\gets 0$\;
	    \For{$i=1,\dots,M$}{
	        $s\gets$ sampled from $\SD$\;
	        $d\gets$ solution to \eqref{P-DM}$(\pi_k, s)$\;
	        \If{$\tau(\pi_k,d) > ub$ and $s_k\not\in\SD_k$}{
	        $ub\gets \tau(\pi_k,d)$\;
	        $s_k\gets s$ \;
            }
	   }
        $\SD_{k+1} \gets \SD_k \cup \{ s_k\}$\; 
		}
\Return $\pi_k, lb$\;
\caption{\textbf{Iterative Improvement Heuristic}}
\label{f-cp-2}
\end{algorithm}

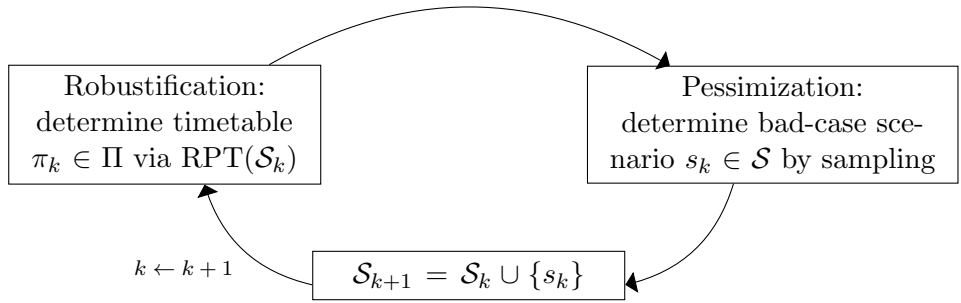
\begin{figure}[h!]
\centering
\begin{tikzpicture}[scale=0.4,auto=left,every node/.style={draw, rectangle, scale=1.0}]
  \node[text width = 12em, text centered, align=center] (pes) at (10,0)  {Pessimization: \\ determine bad-case scenario $s_k\in \SD$ by sampling};
  \node[text width = 10em, text centered, align=center] (opt) at (-10,0)  {Robustification: \\ determine timetable $\pi_k\in \Pi$ via RPT($\SD_k$)};
  \node[text width = 10em, text centered, align=center] (wc) at (0,-5) {$\SD_{k+1} = \SD_k\cup\{s_k\}$};
     \draw[bend left, -triangle 90] (wc) to node[draw=none, fill=none] {\scriptsize{$k\gets k+1$}} (opt);
      \draw[bend left, -triangle 90] (opt) to node[draw=none, fill=none] {} (pes);
          \draw[bend left, -triangle 90] (pes) to node[draw=none] {} (wc);
\end{tikzpicture}
\caption{Iterative Approach for Simplification B}\label{basic-scheme-2}
\end{figure}

Algorithm \ref{f-cp-2} hence gives a lower bound on \eqref{RPT}, but its convergence to the optimal solution of \eqref{RPT} is not guaranteed. The computational experiments show, however, that this scheme -- despite being a heuristic -- already yields good results.

\section{Computational Experiments}\label{sec:exp}
Summarizing the previous chapter, we can conclude that instead of solving \eqref{RPT} we can
\begin{itemize}
\item[(A)] Fix a delay management strategy and solve \eqref{F-RPT} via Algorithm \ref{f-cp}, %to get an upper bound on \eqref{RPT}.
\item[(B)] Find bad-case scenarios and iteratively solve \eqref{RPT}($\SD'$) with increasing $\SD'$, i.e. Algorithm \ref{f-cp-2}.% to get a lower bound on \eqref{RPT}.
\end{itemize}

In the following we describe the setup for determining and evaluating the robust timetables found by \eqref{F-RPT} and \RPTS, and compare them against timetables found by MATCH, i.e., a strong heuristic for solving \eqref{PESP} which has been introduced in \cite{PS16}. MATCH works by setting the buffer times of all waiting and driving times to zero and then heuristically merging line clusters by setting the time differences between them.

Our experiments are carried out on three different datasets which vary in size. \emph{toy} is a small artificial dataset, consisting of 8 stops and 8 edges between them. The \emph{grid} dataset is a 5 $\times$ 5 grid network with 40 edges that was created as a simplified version of the Stuttgart bus network. The final and biggest dataset, \emph{bahn}, consists of 250 stations and 326 edges between them, and represents the ICE network of Germany. Further specifications of the datasets' resulting periodic EAN as well as the parameters of their respective uncertainty sets (cf. Definition \ref{de:unc}) are given in Table \ref{tab:data}. The parameter \textit{passenger cutoff} is used for tractability reasons: The complexity of an EAN is highly influenced by the number of cycles it contains. Every changing activity produces a new cycle but does not influence the feasibility of a timetable (if they have no upper bound, as in our instances). Hence changing activities can be dropped in order to solve the model faster at the price of an inaccurate objective function (cf. \cite{goerigk2017improved} for details). The parameter \emph{passenger cutoff} specfies that all changing activities having \emph{passenger cutoff} or less passengers will be dropped, i.e., ignored by Algorithms \ref{f-cp} and \ref{f-cp-2}. Nevertheless, these changing activities are, of course, considered for the evaluation in Figure \ref{workflow} and also for plotting upper bounds in Figure \ref{plt:bounds}.
\begin{table}[h!]
\centering
\begin{tabular}{c|rrrrrrrr}
 & stops & edges & $|\E|$ & $ |\A|$ & passengers & $\rho$ & $\sigma$ & passenger cutoff\\
\hline
\emph{toy} & 8 & 8 & 88 & 81 & 2622 & 5 & 50 & 0\\
\emph{grid} & 25 & 40 & 260 & 363 & 2546 & 5 & 100 & 10\\
\emph{bahn} & 250 & 326 & 4872 & 6925 & 385868 & 10 & 5000 & 300
\end{tabular}
\caption{Dataset Specifications}\label{tab:data}
\end{table}

We test the three different algorithms according to the workflow in Figure \ref{workflow}.

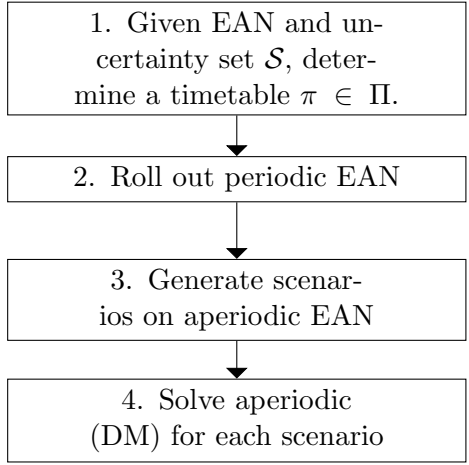
\begin{figure}[h!]
\centering
\hspace*{-8em}
\begin{tikzpicture}[scale=0.4,every node/.style={text centered,draw, text width = 15em,rectangle}]
  \node (tt) at (0, 8)  {1. Given EAN and uncertainty set $\SD$, determine a timetable $\pi\in\Pi$.};
  \node (ro) at (0, 4)  {2. Roll out periodic EAN};
\node (sc) at (0,0) {3. Generate scenarios on aperiodic EAN  };
\node (dm) at (0, -4)  {4. Solve aperiodic (DM) for each scenario};
       \draw[-triangle 90] (tt) to node[left, draw=none, fill=none] {} (ro);
     \draw[-triangle 90] (ro) to node[left, draw=none, fill=none] {} (sc);
      \draw[-triangle 90] (sc) to node[left, draw=none, fill=none] {} (dm);
\end{tikzpicture}
\caption{Workflow for testing the different timetabling algorithms}
\label{workflow}
\end{figure}

Consider the following notes to this workflow:
\begin{enumerate}
\item A timetable is retrieved by either using MATCH, Algorithm \ref{f-cp} or Algorithm \ref{f-cp-2}. For Algorithm \ref{f-cp} we additionally insert a maximal number of  $20$ iterations as well as a time limit of $60$ second for every robustification and every pessimization step. For Algorithm \ref{f-cp-2} we set the maximal number of iterations also to $20$. For retrieving a bad-case scenario for a timetable $\pi$ we sample $100$ scenarios on the periodic EAN and solve each \eqref{P-DM} with a time limit of $10$ seconds. 
\item For the rollout we choose a time horizon of $8$ hours. We hence generate an aperiodic EAN according to Definition \ref{de:rollout}, where the difference $U-L$ is set to $8$ hours.
\item We generate 10 different scenarios for each instance. For generating delays on the aperiodic EAN we, accordingly, multiply $\rho$ by the number of time periods (in our case $8$) used for the rollout. Furthermore we only generate scenarios where the maximal sum of delays is really attained, meaning that our uncertainty set actually looks like
\[
\SD = \left\{d\in \R^{|\A^{dw}|+|\E|} | 0 \leq d_i \leq \sigma, \sum_{i\in \A^{dw}\cup\E} d_i = \rho \tfrac{U-L}{T} \right\}.
\]
\item With the aperiodic EAN we create for each scenario an instance of \eqref{DM} and solve it. We define the two evaluation criteria
\begin{align*}
\text{Nominal Travel Time} &\coloneqq \frac{\sum_{a\in\A} \pi_a  w_a}{\text{\# passengers} \cdot \frac{U-L}{T}},  \\[0.5em]
\text{Delayed Travel Time} &\coloneqq \frac{\tau(\pi, d)}{\text{\# passengers}\cdot \frac{U-L}{T}},
\end{align*}
where the number of passengers is multiplied with the number of time periods in order to take into account the event and activity duplications in the rollout (cf. Definition \ref{de:rollout}).
\end{enumerate}

We then implement Algorithms \ref{f-cp} and \ref{f-cp-2} in LinTim, a software framework for public transport planning (see \cite{lintim}) using Python 3.7 and Gurobi 8.1. The tests are carried out on a standard notebook with 16 GB RAM and an Intel Core i5 processor (2$\times$ 2,3GHz). 

Returning to the speed-up techniques for cutting-plane algorithms that are presented in \cite{PS18} we make the following specifications: After preliminary testing, we set an objective value stopping criterion for F-WC($\pi_k$) of $ub_k$ (i.e., the solver can return the current solution if its objective exceeds $ub_k$). For F-RPT($\SD_k$) we decide to specify a time limit instead of an objective value stopping criterion, and for Algorithm \ref{f-cp-2} we equip (P-DM) with an objective value stopping criterion of $ub$, which sped up the sampling process significantly as it quickly discards ``good-case'' scenarios.

After running the algorithms on all generated scenarios we got $10$ different values for \emph{Delayed Travel Time} for each algorithm-instance pair which are summarized by only considering minimum, maximum and average value of \emph{Delayed Travel Time}. We retrieve the results in Table \ref{results} which are all given in minutes.

\begin{table}[h!]
\centering
\begin{tabular}{c|ccc}
& MATCH & \eqref{F-RPT} & \RPTS\\
\hline
\emph{toy} & 6.0 & 9.1 & 7.1\\
\emph{grid} & 20.4 & 24.6 & 24.1\\
\emph{bahn} & 166.4 & 177.6 & 185.8
\end{tabular}
\caption{Nominal Travel Times in Minutes}\label{nominal}
\end{table}

\begin{table}[h!]
\centering
\begin{subtable}{}
\begin{tabular}{c|ccc}
\emph{toy} & MATCH & \eqref{F-RPT} & \RPTS\\
\hline
Minimum & 12.8 & 11.0 & 10.0 \\
Maximum & 14.6 & 11.8 & 12.1 \\
Average & 13.4 & 11.3 & 11.0
\end{tabular}

\end{subtable}
~\\[1em]
%\begin{table}[h!]
%\centering
\begin{subtable}{}
\begin{tabular}{c|ccc}
\emph{grid} & MATCH & \eqref{F-RPT} & \RPTS\\
\hline
Minimum & 29.3 & 26.9 & 27.8 \\
Maximum & 30.5 & 27.4 & 28.5 \\
Average & 29.9 & 27.3 & 28.3
\end{tabular}
\end{subtable}
~\\[1em]
%\begin{table}[h!]
%\centering
\begin{subtable}{}
\begin{tabular}{c|ccc}
\emph{bahn} & MATCH & \eqref{F-RPT} & \RPTS\\
\hline
Minimum & 205.0 & 204.0 & 201.7 \\
Maximum & 206.5 & 205.4 & 202.7 \\
Average & 205.9 & 204.6 & 202.1
\end{tabular}
\end{subtable}
\caption{Delayed Travel Times in Minutes}\label{results}
\end{table}
Clearly, it can be seen that MATCH yields the best nominal travel times on all three instances. However, after carrying out delay management the travel times of MATCH are the worst among the three algorithms for all instances. \ref{F-RPT} has worse nominal travel times than MATCH, but when taking Delay Management into account the travel times are up to 15\% better than the travel times of MATCH. \RPTS~outperforms \ref{F-RPT} on \emph{toy} and \emph{bahn} with respect to the delayed travel times, but on \emph{grid} \eqref{F-RPT} has better times after DM. We can hence see that our proposed algorithms behave better delayed travel times by paying the price of worse nominal travel times. There are at least two reasons to accept this trade-off to the detriment nominal travel times: The first reason is that in our scenario we only specified realistic delay scenarios that can occur every day: Every single source delay is bounded to be at most 5 (or 10) minutes and also the sum of all source delays was chosen to be reasonable and not extremely high. Thus, only considering the nominal scenario of no delays almost never occurs and hence should not be given so much importance. It is instead more important to find a timetable that has the property of being able to cope with ``small'' delays such as those specified in the uncertainty set.
The second reason is that the propagated delays for MATCH are much larger than for \ref{F-RPT} or \RPTS. Considering \emph{bahn}, MATCH yields an average delay of about $40$ minutes per passenger (\emph{Average Delayed Travel Time} minus \emph{Nominal Travel Time}), whereas  \ref{F-RPT} yields $27$ minutes and \RPTS~only $16$ minutes, cf. Tabular \ref{delays-results}.

\begin{table}[h!]
\centering
\begin{tabular}{c|ccc}
& MATCH & \eqref{F-RPT} & \RPTS\\
\hline
\emph{toy} & 7.4 & 2.2 & 3.9\\
\emph{grid} & 9.5 & 2.7 & 4.2 \\
\emph{bahn} & 39.5 & 27.0 & 16.3
\end{tabular}
\caption{Average Passenger Delay in Minutes}
\label{delays-results}
\end{table}

It is arguable that the size of the propagated delays has even higher importance than the Nominal Travel Time: As mentioned in the introduction, people complain much more about delays and not so much about long travel times. Hence, it would even be reasonable to assign the delays with an additional weight factor in order to model these preferences which reinforces the quality of the timetables found by \ref{F-RPT} and \RPTS. Thus, traditional periodic timetabling chooses a timetable according to Table \ref{nominal}, whereas it makes more sense to make the choice based on the values of Table \ref{results}. In order to avoid passenger complaints, the decision can even be based on the values of Table \ref{delays-results}. Note that the consideration of the different timetable objective leads to a discussion on the compatibility of different evaluation functions which is analyzed in \cite{hartleb2019good}.

Another interesting observation is the improvement of the bounds of the algorithms for models \eqref{F-RPT} and \RPTS. The bounds depict the values of $lb_k$ and $ub_k$ from Algorithm \ref{f-cp} for \eqref{F-RPT} and the values of $lb$ and $ub$ from Algorithm \ref{f-cp-2} for \RPTS, each divided by (\# passengers) to get average travel times. Note that $ub$ of \RPTS~can lie below $lb$ (as for the \emph{toy} example) since it is not a real upper bound as Algorithm \ref{f-cp-2} determines it by scenario sampling.

\begin{figure}[h!]
\centering
\includegraphics[scale=1.0]{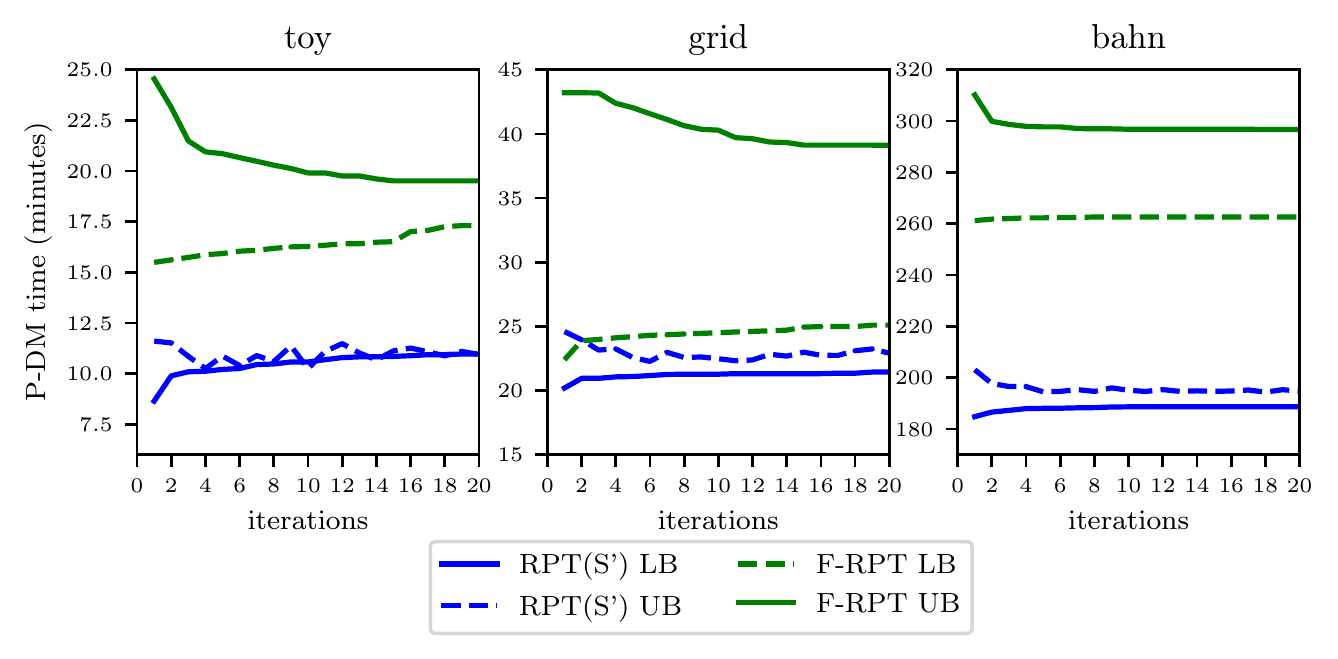}
\caption{Solution improvements}
\label{plt:bounds}
\end{figure}

Unfortunately, we see that the value of \eqref{RPT}, which has to lie somewhere between the lower bound of \RPTS~and the upper bound of \eqref{F-RPT}, cannot be pinned down to a small interval. Nevertheless, one interesting observation is there to be made: We started with the timetable found by MATCH in all instances, meaning that the upper bound in every first iteration corresponds to the model's evaluation of the MATCH timetable. Interestingly, there is a significant decrease in the upper bound from iteration 1 to iteration 2 in all instances for both models (except for \eqref{F-RPT} on \emph{grid}). Hence, the models find quite bad scenarios for the MATCH timetable, leading to the conclusion that MATCH generates actually a quite non-robust timetable. Furthermore it can be seen that after 20 iterations the bounds do not change so much anymore. This is especially true for \emph{bahn} and can be explained by the sheer size of the resulting problem. Thus, the models, despite being a simplification of \eqref{RPT}, are still computationally challenging and there might also be room for improvements regarding algorithm design and parameter tuning. Still, the timetables we have found already yield an improvement with respect to robustness in comparison to MATCH.

Finally, if we compare the values for travel times found by our algorithms against the travel times estimated by the workflow we can see that the estimated travel times fall between the lower bounds for \RPTS~and upper bounds for \eqref{F-RPT}, but much closer to the values of \RPTS. Hence, we can deduce that after a few iterations \RPTS~yields a good approximation of the delayed travel times, whereas \eqref{F-RPT} provides merely an upper bound and is not very close to the values of the evaluated timetables in Table \ref{results}. We conclude that even though the two algorithms introduced in this paper may not serve to give precise bounds on \eqref{RPT}, but they reliably find robust timetables.

\section{Conclusion and Outlook}\label{sec:con}
In this paper we have introduced a new model \eqref{RPT} for finding robust periodic timetables. Due to its high difficulty we introduced two relaxed versions of model \eqref{RPT}, namely \eqref{F-RPT}, which assumes a fixed delay management strategy, and \RPTS, which solves the robust timetabling problem for a finite uncertainty set. For each of the two models we presented a solution algorithm, i.e., Algorithm \ref{f-cp} for \eqref{F-RPT} and Algorithm \ref{f-cp-2} for \RPTS~which we tested computationally against MATCH, a standard \eqref{PESP} algorithm. The computational experiments show that for the timetables found by MATCH even small delays can induce huge passenger delays, raising concerns about the point of finding timetables that merely perform well in the absence of delays. Our proposed models, on the other hand, can cope significantly better with delays, but have worse nominal travel times -- a trade-off that seems reasonable when considering customer opinions of public transport systems.

Further research can be carried out in different directions. First and foremost we can further improve the proposed solution algorithms: Obviously, there are a number of additional parameters involved in the implementation that can be optimized (or learned). What is more, hybrid strategies of the two algorithms are also possible to implement and might improve the results. Different relaxation strategies, for example by dualizing the linear relaxation of \eqref{P-DM}, could also yield to strong solution algorithms. In general, it is possible to implement different delay management strategies instead of the no-wait strategy, as long as the new strategy can be formulated in the integer program. It would be very interesting to investigate the outcome of a different delay management strategy.
We could, of course, also add more detail to the model. Adding vehicle schedules by using turnaround activities in the EAN could be an interesting starting point. These turnaround activities would, however, be problematic for the problem \eqref{F-WC} as delays also would propagate along turnaround activities. Hence we would get no events $\Es$ and would have to find a different formulation to get rid of the inner minimization problem in order to simplify \eqref{RPT}.
On the theoretical side it would be interesting to analyze the periodic delay management model \eqref{P-DM} in more detail: Especially, finding a concrete bound between the two models \eqref{DM} and \eqref{P-DM} could further fortify the idea of considering delay management on periodic networks.
Finally, defining a slightly different problem might also be a good way to craft a robustness measure: In \eqref{RPT} it might be better not to use the maximum among all scenarios but to sum over all scenarios in order to get a better robustness measure instead. The advantage would be that the maximum vanishes and \eqref{RPT} would reduce to a quite difficult minimization problem.

All things considered, there are many ways to continue the proposed path of finding robust periodic timetables by including delay management. We hope that it will become more popular to consider potential delays and not only nominal travel times when planning a periodic timetable.

%\newpage
%\bibliography{eigen,rob-pesp}
\printbibliography
\end{document}